\documentclass{amsart}

\newif\ifdraft\draftfalse
\newif\ifcite\citefalse
\newif\ifblow\blowfalse

\makeatletter
\@namedef{subjclassname@2020}{\textup{2020} Mathematics Subject Classification}
\makeatother

\usepackage{amsfonts,amssymb, amsmath}
\usepackage{oldgerm}
\usepackage{url}
\usepackage{amscd}
\usepackage{mathrsfs} 
\usepackage[hyperindex]{hyperref}  
\usepackage[alphabetic]{amsrefs}
\ifdraft\ifcite\usepackage{showkeys}\else\usepackage[notcite,notref]{showkeys}\fi\fi

\newtheorem{proposition}[equation]{Proposition}
\newtheorem{theorem}[equation]{Theorem}

\theoremstyle{definition}

\theoremstyle{remark}

\newtheorem{remark}[equation]{Remark}

\newtheorem{rem}[equation]{Remark}
\newtheorem{question}[equation]{Question}

\numberwithin{equation}{section}

\def\bc{\begin{cases}}
\def\ec{\end{cases}}

\def\a{\alpha}

\newcommand\om{\omega}

\def\Bbb{\mathbb}

\def\t{\tilde}

\def\ch{{\mathcal H}}

\def\bc{{\mathbb C}}

\def\br{{\mathbb R}}

\def\bz{{\mathbb Z}}

\def\er{\eqref}

\def\lp2{L_pH_{2p}}

\def\beq{\begin{equation}}
\def\eeq{\end{equation}}
\def\bal{\begin{align*}}
\def\eal{\end{align*}}
\def\baln{\begin{align}}
\def\ealn{\end{align}}

\def\beg{\begin{gather*}}
\def\eng{\end{gather*}}
\def\bqu{\begin{question}}
\def\equ{\end{question}}

\def\ban{\begin{proof}[Answer]}
\def\ean{\end{proof}}

\def\ra{\Rightarrow}

\def\p{\partial}

\def\on{\operatorname}

\def\bqu{\begin{question}}
\def\equ{\end{question}}

\def\0110{\begin{matrix} 0 & 1\\1&0\end{matrix}}

\def\t{\tilde}

\def\fg{\mathfrak{g}}
\def\fh{\mathfrak{h}}

\def\fl{\mathfrak{l}}

\def\fn{\mathfrak{n}}

\def\fs{\mathfrak{s}}

\def\ban{\begin{proof}[Answer]}
\def\ean{\end{proof}}

\def\ben{\begin{equation}}
\def\een{\end{equation}}

\def\la{\langle}
\def\ra{\rangle}

\def\j1{{(j+1)}}

\def\f32{{}_3F_2}

\newcommand{\fhr}{{\mathfrak h}^{\rm reg}}

\begin{document}

\title
{Blowup masses of Toda systems corresponding to the Weyl groups}

\author{Zhaohu Nie}
\email{zhaohu.nie@usu.edu}
\address{Department of Mathematics and Statistics, Utah State University, Logan, UT 84322-3900, USA}


\begin{abstract} 
Toda systems are generalizations of the Liouville equation to systems using simple Lie algebras. We study the blowup phenomena of their solutions by giving concrete examples demonstrating blowup masses corresponding to 
the Weyl groups. 
\end{abstract}

\maketitle

\section{Introduction}
As a motivation for the subject of study in this paper, blowup phenomena of Toda systems in dimension 2, we first recall the classical scalar curvature equations in dimensions $\geq 3$ and dimension 2, and their corresponding blowup masses. 

Conformal geometry and the related Yamabe and prescribed scalar curvature problems \cite{SY} are classical topics in differential geometry and analysis of PDEs. Let $(M, g)$ be a Riemannian manifold, and $\tilde g$ a conformal metric of $g$. The corresponding scalar curvatures $R$ and $\t R$ are related as follows. If $\dim M = n\geq 3$, write 
$$
\t g = u^{\frac{4}{n-2}} g, \quad u\in C^\infty(M),\ u>0,
$$
then 
$$
4\frac{n-1}{n-2} \Delta_g u  + \t R u^{\frac{n+2}{n-2}} = Ru,
$$
where $\Delta_g$ is the Beltrami-Laplace operator of $g$. 
In particular, if $g_0$ is the Euclidean metric on $\br^n$ with $R=0$, and we would like $u^{\frac{4}{n-2}} g_0$ to have the scalar curvature of the standard round metric $g_1$ on the sphere $S^n$ with $\t R = n(n-1)$, we have the equation, on $\br^n$, 
\begin{equation}\label{sceq}
\Delta u + \frac{n(n-2)}{4} u^{\frac{n+2}{n-2}} = 0,
\end{equation}
where $\Delta$ is the usual Laplacian on $\br^n$. 
   
From the stereographic projection followed by a translation and a scaling, some explicit solutions can be computed to be 
\begin{equation}\label{ula}
u_{\lambda, a}(x) = \left(\frac{2\lambda}{1+\lambda^2 |x-a|^2}\right)^{\frac{n-2}{2}}, \quad \lambda>0,\ a\in \br^n.
\end{equation}
It is proved in \cites{GNN, CGS, S} that these are the only positive solutions of \eqref{sceq} on $\br^n$. 

For these solutions, we have
\begin{equation}\label{massn}
\int_{\br^n} u^{\frac{2n}{n-2}}_{\lambda, a} = \alpha_n = \frac{2\pi^{\frac{n+1}{2}}}{\Gamma(\frac{n+1}{2})},
\end{equation}
the surface area of $S^n$. 
In \eqref{ula}, we see that as $\lambda\to \infty$, $u_{\lambda, a}$ get concentrated at $x=a$, and we 
actually have 
\begin{equation}\label{scalar local}
\lim_{r\to 0}\lim_{\lambda\to \infty} \int_{B_r(a)} u_{\lambda, a}^{\frac{2n}{n-2}} = \alpha_n,
\end{equation}
where $B_r(a)$ is the ball centered at $a$ with radius $r>0$. For this blowup behavior, we call the number $\alpha_n$ the blowup mass. 




For the corresponding conformal geometry in dimension 2, it is better to use
$$
\t g = e^{2v} g,\quad v\in C^\infty(M),
$$
and the Gaussian curvature $K=\frac{R}{2}$. Then the equation corresponding to \eqref{sceq} is 
$$
\Delta_g v + \t K e^{2v} = K,
$$
which is called the Liouville equation. 

In the dimension 2 case, it is best to adapt the complex point of view by considering $\br^2$ as $\bc$ with complex coordinate $z$, and use the Riemann sphere $\bc P^1$ endowed with the Fubini-Study Hermitian metric $\frac{dz\otimes d\bar z}{(1+|z|^2)^2}$, whose Gaussian curvature is $4$.

Then, for $e^{2v} g_0$ on $\br^2$ to have Gaussian curvature $4$ as $\bc P^1$, the equation is 
\begin{equation}\label{Liouv}
\Delta v + 4e^{2v} = 0
\end{equation}
on $\br^2$. The solutions with finite energy, that is, such that
$$
\int_{\br^2}e^{2v}<\infty
$$
are known \cite{CL0} to be 
\beq\label{liou sol}
v_{\lambda, a}(x) = \log \frac{\lambda}{1+\lambda^2|x-a|^2},\quad \lambda>0,\ a\in \br^2.
\eeq
Similarly to \eqref{massn}, we have 
$$
\int_{\br^2} 4e^{2v_{\lambda, a}} = 4\pi.
$$

In this dimension 2 case, we can also allow some singularity at the origin. Let $\gamma>-1$, and denote $\mu=\gamma+1>0$. We require that $e^{2v}g_0$ have a conical singularity at the origin with cone angle $2\mu\pi$ and have constant Gaussian curvature 4 elsewhere. The generalization of \eqref{Liouv} to this case is 
\beq\label{compare Liouv}
\Delta v + 4e^{2v} = 2\pi\gamma \delta_0,
\eeq
where $\delta_0$ is the Dirac delta measure at the origin. 
The general solutions \cite{PT} are now 
$$
v_{\lambda,c}(z) = \log \frac{\lambda|z|^\gamma}{1 + \lambda^2|\frac{z^\mu}{\mu}+c|^2},\quad \lambda>0,\ c\in \bc,
$$ 
where $c=0$ if $\gamma\notin \bz_{\geq 0}$ (for the solution to be defined on $\br^2\backslash\{0\}$). For this solution, we have 
\begin{equation}\label{liouv4}
\int_{\br^2}4e^{2v_{\lambda,c}} = 4\mu\pi,
\end{equation}
which is the Gauss-Bonnet theorem for $\bc P^1$ with the same conical singularities at 0 and $\infty$. 

For the purpose of this paper, even if $\gamma\in \bz_{\geq 0}$, we will only consider the special solution of $c=0$ from now on and write $v_\lambda=v_{\lambda,0}$. 
By a direct calculation, we have the blowup behavior
\begin{equation}\label{Liouv local}
\lim_{r\to 0}\lim_{\lambda\to \infty} \int_{B_r(0)} e^{2v_\lambda} = \mu\pi.
\end{equation}

Note that in these cases of scalar curvature equations and Liouville equation, the local blowup masses in \eqref{scalar local} and \eqref{Liouv local} are the same as the global masses in \eqref{massn} and \eqref{liouv4}. 



Now, we introduce the subject of study in this paper. For a complex simple Lie algebra $\fg$ of rank $n$ with Cartan matrix $(a_{ij})$, we consider the following Toda system of real-valued unknowns $u_i$ on the plane $\br^2$ with singularities at the origin and with finite energies
\begin{equation}\label{Toda}
\begin{cases}
\displaystyle
\Delta u_i + 4\sum_{j=1}^n a_{ij} e^{u_j} = 4\pi \gamma_i \delta_0 \quad\text{on }\br^2, & \gamma_i> -1,\\
\displaystyle
\int_{\br^2} e^{u_i} \,dx <\infty, & 1\leq i\leq n.
\end{cases}
\end{equation}

It is  convenient to introduce 
$$
U_i = \sum_{j=1}^n a^{ij} u_j,\quad 1\leq i\leq n, 
$$
where $(a^{ij})=(a_{ij})^{-1}$ is the inverse Cartan matrix. Then $u_i = \sum_{j=1}^n a_{ij}U_j$. 
Using $\Delta = 4 \frac{\p}{\p z} \frac{\p}{\p \bar z}$ with $\frac{\p}{\p z}= \frac{1}{2}(\frac{\p}{\p x} - i\frac{\p}{\p y})$ and $\frac{\p}{\p \bar z}= \frac{1}{2}(\frac{\p}{\p x} + i\frac{\p}{\p y})$, 
the Toda system \eqref{Toda} becomes
$$
 U_{i,z\bar z} + e^{u_i} = \pi\gamma^i \delta_0,
$$
with $\gamma^i = \sum_{j=1}^n a^{ij} \gamma_j$.
In the Liouville case  \eqref{compare Liouv}, $\fg=\fs\fl_2$, $a_{11}=2$, $v=U_1$, $2v=u_1$, and $\gamma=\gamma_1$.

The classification of solutions to \eqref{Toda} on $\br^2\backslash\{0\}$  is done in \cite{KLNW}, and we also have the quantization result for their masses. These results are the basis of the current paper, so we state them now. In Section 2, we will state the Lie-theoretic setup in more detail. 

For the complex simple Lie algebra $\fg$ of rank $n$, we choose a Cartan subalgebra $\fh$, so we have a root space decomposition 
$$\fg = \fh \oplus \bigoplus_{\alpha \in \Delta} \fg_\alpha,$$
where $\Delta$ is the root system, and $\fg_\a$ is the root space for $\alpha\in \Delta$. 
Recall that the roots are linear functions on the Cartan subalgebra $\mathfrak{h},$ and for $X_{\alpha} \in \mathfrak{g}_{\alpha}$ and $H \in \mathfrak{h},$
we have $[H, X_{\alpha}] = \alpha(H)X_{\alpha}.$ 
We have $\dim_\bc \fh=n$ and $\dim_{\mathbb{C}} \mathfrak{g}_{\alpha} = 1$ for all $\alpha \in \Delta$. 

The real Cartan subalgebra is 
\begin{equation*}
\fh_0 = \{H\in \fh \,|\, \a(H)\in \br,\ \forall \a\in \Delta\}.
\end{equation*}
It is know that $\dim_\br \fh_0=n$. 
Let $\mathfrak{h}_0^{\prime} := \mbox{Hom}(\mathfrak{h},\mathbb{R})$ be the real dual space of $\mathfrak{h}_0$, and $\langle \cdot,\cdot\rangle$ the pairing of $\mathfrak{h}_0^{\prime}$ and $\mathfrak{h}_0$.  
 Clearly $\Delta\subset \fh_0'$. 

For the root system $\Delta$, we choose a decomposition $\Delta = \Delta^+\cup \Delta^-$ into positive and negative roots. Let 
$$
\Pi=\{\alpha_1,\dots,\alpha_n\}\subset \Delta^+
$$
be the corresponding simple roots. 

We choose a Chevalley basis 
\beq\label{chev basis}
\{e_{\alpha} \in \mathfrak{g}_{\alpha}, \alpha \in \Delta;\ h_{\alpha_i} \in \mathfrak{h}_0, 1 \leq i \leq n\}
\eeq
for $\fg$ (see \cite{Humphreys}*{Theorem 25.2} and \cite{Knapp}*{Theorem 6.6}).

Introduce 
$\mu_i = \gamma_i + 1 >0,\ 1\leq i\leq n,$
and define 
$w_0\in \fh_0$ as the unique element such that 
\begin{equation}
\label{def w0}
\la \alpha_i, w_0\ra = \mu_i, \quad 1\leq i\leq n. 
\end{equation}

Let $G$ be a connected complex Lie group whose Lie algebra is $\fg$. The classical Iwasawa decomposition \cite{Knapp}*{Theorem 6.46} says that $G=KAN$  with the subgroups $K$ maximally compact, $A$ abelian and $N$ nilpotent. It is well known that the groups $A$ and $N$ are {\it simply connected.} 

Let $V_i$ be the $i$-th fundamental representation of $\mathfrak{g}$ and denote the highest weight vector by $|i \rangle.$ Let $V_i^*$ be the dual right representation of $\mathfrak{g}$ and choose the lowest weight vector denoted by by $\langle i |$ so that the pairing 
$\langle i||i\rangle = 1.$

An essential role in the solution of \eqref{Toda} is played by the following. 
Let $\Phi: \bc\backslash \br_{\leq 0}\to N$ be the unique holomorphic solution of 
\begin{equation}\label{Phi}
\begin{cases}
\displaystyle
\Phi^{-1} \Phi_z = \sum_{i=1}^n z^{\gamma_i} e_{-\alpha_i} \quad \text{on }\bc\backslash \br_{\leq 0},\\
\displaystyle\lim_{z\to 0}\Phi(z) = Id,
\end{cases}
\end{equation}
where $Id\in G$ is the identity element (the limit exists because $\gamma_i>-1$). 

Then \cite{KLNW}*{Theorem 1.5} proves that the solutions to \eqref{Toda} are 
\beq
U_i= 2\gamma^i \log |z| -\log\la i|\Phi^* C^*\Lambda^2 C \Phi|i\ra,
\label{good oldie}
\eeq
where $C \in N$ and $\Lambda \in A$, and ${}^*$ is defined in \eqref{* on gp}. In this paper, we will only use $C=Id\in N$. 

Furthermore, we have the global masses \cite{KLNW}*{Theorem 1.1}
\beq\label{global mass}
\frac{1}{\pi}\int_{\br^2} e^{u_i} =  \langle \omega_i - \kappa \omega_i, w_0 \rangle,
\eeq
where $\kappa$ is the longest element in the Weyl group of $\fg$ and $\omega_i$ is the $i$th fundamental weight. 
  
Unlike the cases of the scalar curvature equations \eqref{scalar local} and the Liouville equation \eqref{Liouv local}, one key new phenomenon for the Toda system is that the local blowup masses are more varied than the global mass \eqref{global mass}. 

A general expectation is that the set of blowup masses 
is closely related 
to the Weyl group $W$ of $\fg$. This phenomenon has been studied in  \cite{LYZ} for Lie algebras of types $A$, $B$ and $C$ and even in some cases of affine Lie algebras in \cite{CNY}.  
The purpose of this paper is to construct concrete examples that demonstrate such local masses corresponding to the elements of the Weyl group as 
$$
\big\{(\la \omega_1 - \tau\omega_1, w_0\ra, \cdots, \la \omega_n - \tau\omega_n, w_0\ra)\,|\, \tau\in W\big\}.
$$

Let 
\beq\label{dom cham}
C_0 = \{H\in \fh_0 \,|\, \la\a_i, H\ra >0,\, 1\leq i\leq n\}
\eeq
be 
the dominant Weyl chamber. 

Here is our main result, which demonstrates the role of the Weyl group in the blowup behavior.


\begin{theorem}\label{main thm}
Let $\tau\in W$ be a Weyl group element, and let $H\in \tau C_0$, the image of $C_0$ under $\tau$. For $\lambda>0$, consider the family of solutions 
\beq\label{mysol}
U_i^{\lambda H} = 2\gamma^i \log |z| - \log \la i | \Phi^* \exp(2\lambda H) \Phi |i\ra,
\eeq
and the corresponding 
$$
u_i^{\lambda H} = \sum a_{ij} U_j^{\lambda H}
$$
which satisfy 
\beq\label{the eq}
\Delta U_i^{\lambda H} + 4e^{u_i^{\lambda H}} = 4\pi \gamma^i \delta_0.
\eeq

Then the local masses, divided by $\pi$, are 
\beq\label{local mass}
\sigma_i = \frac{1}{\pi} \lim_{r\to 0}\lim_{\lambda\to \infty} \int_{B_r(0)} e^{u_i^{\lambda H}} = \la \omega_i-\tau\omega_i, w_0\ra.
\eeq
\end{theorem}


Amazingly, in one key step \eqref{c>0} of the proof, we need to use a purely Lie theoretic result from a recent paper \cite{BKK}.

\section{Lie-theoretic background and results}

To prove our main theorem, we need to study the $\Phi$ in \eqref{Phi} in more detail. It was studied in \cite{KLNW} following \cite{KosToda} and served as a major tool in the proof, especially of \eqref{global mass}. Here we establish one more result in Proposition \ref{alg char} to bring the relationship to \cite{KosToda} even closer, and this enables us to utilize a Lie-theoretic result in Proposition \ref{nonzero coeff} about it. 

Our basic references for Lie theory and representation theory are \cites{Knapp, Humphreys} and \cite{LSbook}. 
Recall that we have chosen the Chevalley basis \eqref{chev basis}. As a result, 
we have 
\begin{equation*}
[e_{\alpha_i}, e_{-\alpha_j}] = \delta_{ij}h_{\alpha_i},\quad  \alpha_i(h_{\alpha_j})=a_{ij},\quad 1\leq i, j\leq n,
\end{equation*}
where $(a_{ij})$ is the Cartan matrix. 

For $E_j = \sum_{k=1}^n a^{kj}h_{\a_k}\in \fh_0$, where $(a^{ij})$ is the inverse Cartan matrix, we have 
	\begin{equation}\label{Ej}
	\a_i(E_j) = \delta_{ij},\quad 1\leq i, j\leq n.
	\end{equation}
Therefore, $w_0\in \fh$ in \eqref{def w0} is given as 
$$w_0= \sum_{i=1}^n\mu_i E_i.$$

We introduce the following subalgebras of $\mathfrak{g}$
\begin{equation}\label{nn}
\mathfrak{n}_- = \oplus_{\alpha \in \Delta^+} \mathfrak{g}_{-\alpha},\quad \mathfrak{n}_+ = \oplus_{\alpha \in \Delta^+} \mathfrak{g}_{\alpha}.
\end{equation}

Let $\ch$ be the Cartan subgroup of $G$ corresponding to the complex Cartan subalgebra $\fh$. Denote the subgroups of $G$ corresponding to $\fn=\fn_-$ and $\fn_+$ by $N=N_-$ and $N_+$. Then by the Gauss decomposition (see \cite{KosToda}*{Eq. (2.4.4)}), there exists an open and dense subset $G_r\subset G$, called the regular part, such that 
\begin{equation}\label{Gdec}
G_r = N_- \ch N_+.
\end{equation}
Clearly $G_r$ contains the identity element of $G$. 

Let ${}^* :\fg \to \fg$ be the conjugate-linear transformation defined in terms of the Chevalley basis \eqref{chev basis} by  
\begin{equation}\label{dual} 
{}^*|_{\fh_0} = \text{Id}|_{\fh_0}, 
\quad e_\a^* = e_{-\a},\ \a\in \Delta.
\end{equation}
We also let ${}^*:G\to G$ denote the lift of ${}^*$ on $\fg$ \cite{KosToda}*{eq. (3.1.9)} such that 
\beq\label{* on gp}
(\exp x)^* = \exp x^*, \quad x\in \fg,
\eeq
where $\exp:\fg\to G$ is the exponential map. Note that if $a, b\in G$, one has $(ab)^*=b^*a^*$ by \cite{KosToda}*{eq. (3.1.8)}. 

\begin{remark} For $G=SL_{n+1}\bc$, ${}^*$ is the usual conjugate transpose for a standard choice of Chevalley basis. 
\end{remark}

We also have the following Iwasawa decomposition \cite{Knapp}*{Theorem 6.46}  
\begin{equation}\label{Iwasawa}
G = KAN,
\end{equation}
where $K$ is a maximal compact subgroup of $G$,
$A$ is an abelian subgroup corresponding to $\fh_0$, and $N$ is as above corresponding to $\fn=\fn_-$ . 
The subgroups $A$ and $N$ are simply-connected. An element $F$ belongs to $K$ if and only if $F^* F = Id$.

Define the following
\begin{align}
f &= \sum_{i=1}^n e_{-\a_i}\in \fg_{-1} := \oplus_{i=1}^n \fg_{-\a_i},\label{def f}\\
\zeta &= \sum_{i=1}^n z^{\gamma_i} e_{-\a_i}: \bc\backslash \br_{\leq 0}\to \fg_{-1}, \label{the zeta}\\
\xi &= z\zeta = \sum_{i=1}^n z^{\mu_i} e_{-\a_i} .\label{the xi}
\end{align}
Then \eqref{Phi} becomes $\Phi^{-1} \Phi_z = \zeta$. 

The Weyl group $W$ of a Lie algebra $\fg$ is the finite group generated by the reflections in the simple roots on $\fh'_0$ 
$$
s_i (\beta) = \beta - \frac{2(\beta, \a_i)}{(\a_i, \a_i)} \a_i,\quad \beta\in \fh_0',\ 1\leq i\leq n,
$$
where $(\cdot, \cdot)$ is the inner product on $\fh_0'$ induced from the Killing form on $\fh_0$. 
The Weyl group $W$ maps the root system $\Delta$ to itself. 

There is a dual action of $W$ on $\fh_0$ such that 
\beq\label{W on h}
\la \tau(\beta), \tau(x)\ra = \la \beta, x\ra,\quad \tau\in W,\ \beta\in \fh_0',\ x\in \fh_0.
\eeq

There are lifts of $\tau\in W$ to $\bar \tau \in N_\fh(G)=\{g\in G\,|\, \on{Ad}_g\fh \subset \fh\}$ (see \cite{BKK}*{\S 2.1}) such that 
$$
\tau x = \on{Ad}_{\bar \tau} x,\quad \tau\in W,\ x\in \fh.
$$
Here $\on{Ad}$ is the adjoint action of $G$ on $\fg$. 
For example, one can lift 
the reflection $s_{\a_i}$ in a simple root to 
$$
\bar s_{\a_i} = \exp(e_{\a_i})\exp(-e_{-\a_i}) \exp(e_{\a_i}),
$$
so one can lift all elements in $W$.

Let 
$$\fh^{\rm reg} = \{H\in \fh\,|\, \alpha(H)\neq 0,\ \forall \alpha\in \Delta\}.$$
Clearly $w_0\in \fh^{\rm reg}$ by \eqref{def w0}. Furthermore, if $\tau\in W$ and $x\in \fhr$, then 
\beq\label{still reg}
\tau x\in \fh^{\rm reg},\quad \text{ as }
\la\a, \tau x\ra = \la \tau^{-1}\a, x\ra \neq 0,
\eeq
by \er{W on h} and $\tau^{-1} \a\in \Delta$. 

By \cite{KosToda}*{Lemma 3.5.1} and also  \cite{BKK}*{\S 8.1} (which attributes it to Bourbaki), for each $x\in \fh^{\rm reg}$, there is a unique element $n_x\in N_-$ such that $\on{Ad}_{n_x^{-1}} x = x - f$. 
We also write $\on{Ad}_g x$ as $g x g^{-1}$ for $g\in G$ and $x\in \fg$. 

We have the following algebraic characterization of the $\Phi$ in \eqref{Phi}. 
\begin{proposition}\label{alg char} The unique map $\Psi : \bc\setminus \br_{\leq 0}\to N=N_-$ such that 
\begin{equation}\label{def psi}
\on{Ad}_{\Psi^{-1}} w_0  = w_0 - \xi,
\end{equation}
where $\xi$ is defined in \eqref{the xi}, 
is $\Psi = \Phi$.
\end{proposition}


\begin{proof}
We show that $\Psi$ is a holomorphic solution of the equation \eqref{Phi} defining $\Phi$, then it follows that $\Psi=\Phi$. 
From \eqref{def psi}, it is clear that $\lim_{z\to 0}\Psi = Id$ since $\lim_{z\to 0}\xi = 0$. 

Let $\psi$ be the unique element in $N_-$ such that 
\beq\label{lil psi}
\on{Ad}_{\psi^{-1}}w_0 = w_0 - f.
\eeq

Let 
\begin{equation}\label{choose Q}
Q = \exp\Big( \sum_{j=1}^n \big(\log z^{\mu_j}\big) E_j \Big): \bc\setminus \br_{\leq 0}\to \ch,
\end{equation}
where $\ch$ is the Cartan subgroup of $G$. 
Note that a single-valued branch of log can be chosen since $\bc\setminus \br_{\leq 0}$ is simply-connected. 
We have
\begin{equation*} 
\begin{split}
\on{Ad}_{Q^{-1}}e_{-\a_i} &= \exp\Big( \sum_{j=1}^n \big(\log z^{\mu_j}\big) \alpha_i(E_j) \Big) e_{-\a_i}  \\
&= \exp\big( \log z^{\mu_i}\big) e_{-\a_i}= z^{\mu_i} e_{-\a_i},
\end{split}
\end{equation*}
in view of \eqref{Ej}. 
Therefore, 
\beq\label{new conj}
\on{Ad}_{Q^{-1}}f = \xi
\eeq
by \eqref{def f} and \eqref{the xi}. 

Then $Q^{-1}\psi Q\in N_-$ is the unique element $\Psi$ in \eqref{def psi}, since 
\begin{align*}
\on{Ad}_{(Q^{-1}\psi Q)^{-1}} w_0 &= \on{Ad}_{Q^{-1}}\on{Ad}_{\psi^{-1}}\on{Ad}_Q w_0\\
&=  \on{Ad}_{Q^{-1}}\on{Ad}_{\psi^{-1}} w_0\\
&= \on{Ad}_{Q^{-1}}(w_0-f) = w_0-\xi,
\end{align*}
where the equalities follow from $Q\in \ch$ and $w_0\in \fh$, \eqref{lil psi} and \eqref{new conj}. Note that $\Psi=Q^{-1}\psi Q: \bc\setminus \br_{\leq 0}\to N_-$ is evidently holomorphic. 

From \eqref{choose Q}, we have 
\beq\label{qz}
Q_zQ^{-1}=Q^{-1}Q_z=\frac{1}{z}\sum_{j=1}^n \mu_j E_j =\frac{1}{z}w_0.
\eeq
Therefore, \eqref{qz}, \eqref{lil psi} and \eqref{new conj} give 
\begin{align*}
\Psi^{-1}\Psi_z &= (Q^{-1}\psi^{-1} Q)(-Q^{-1} Q_z Q^{-1} \psi Q) + Q^{-1} Q_z\\
&= -\frac{1}{z}Q^{-1}(\psi^{-1} w_0 \psi) Q + \frac{1}{z} w_0\\
&= -\frac{1}{z} Q^{-1}(w_0-f)Q + \frac{1}{z} w_0\\
&= -\frac{1}{z}(w_0 - \xi) + \frac{1}{z} w_0\\
&= \frac{1}{z}\xi = \zeta,
\end{align*}
so $\Psi$ satisfies \eqref{Phi}. Therefore, $\Psi=\Phi$. 
\end{proof}



The following result plays an important role in the proof of our main theorem. It is a modified version of \cite{BKK}*{Prop 8.1}.\footnote{I thank Joel Kamnitzer for answering my question on MathOverflow.}


\begin{proposition}\label{nonzero coeff}
Let $x\in \fh^{\rm reg}\cap \fh_0$, and $\psi\in N_-$ be the unique element such that 
\beq\label{psix}
\psi^{-1} x \psi = x - f.
\eeq
Let $\tau\in W$. 
Then 
$$
\bar \tau^{*} \psi \in G_r, 
$$
that is, it has Gauss decomposition \eqref{Gdec}. 
\end{proposition}

\begin{proof} By $\tau^{-1} x\in \fhr$ \eqref{still reg}, there exists a unique $\psi_1\in N_-$ such that
\beq\label{psi1tx}
\psi_1^{-1} (\tau^{-1} x) \psi_1 = \tau^{-1} x -f.
\eeq

Applying $*$ to the two identities \eqref{psix} and \eqref{psi1tx}, we have, by \eqref{dual},
\begin{align*}
\psi^* x (\psi^*)^{-1} &= x - e,\\
\psi_1^* (\tau^{-1} x) (\psi^*_1)^{-1} &= \tau^{-1} x - e,
\end{align*}
where $e=\sum_{i=1}^n e_{\a_i}$. 

Therefore $\psi^*=n_x\in N_+,\ \psi_1^*=n_{\tau^{-1} x}\in N_+$ in the notation of \cite{BKK}*{\S 8.1} such that 
\begin{align*}
n_x x n_x^{-1} &= x-e,\\
n_{\tau^{-1} x} (\tau^{-1} x) n_{\tau^{-1} x}^{-1} &= \tau^{-1} x-e,
\end{align*}
(Note that their $\dot e=\sum_{i=1}^n \dot e_i$ where each $\dot e_i$ is a non-zero root vector in $\fg_{\a_i}$. So our $-e$ is one such $\dot e$.)

By their Proposition 8.1, 
there exist $y\in N_-$ and $t\in \ch$ such that 
$$
n_{\tau^{-1} x} = yn_x\bar\tau t, \quad \text{i.e. } \psi_1^* = y \psi^* \bar \tau t.
$$

Applying $*$ to the above equation, we get 
$$
\psi_1 = t^* \bar \tau^{*} \psi y^*.
$$
Therefore, with $h=(t^*)^{-1}\in \ch$ and $\t y=(y^*)^{-1}\in N_+$, we have
$$
\bar\tau^{*} \psi = h\psi_1 \t y = (h \psi_1 h^{-1}) h \t y\in G_r.
$$
\end{proof}


For applications in the next section, let's review some representation theory. 
The integral weight lattice of $\fg$ is $\Lambda_W = \{\beta\in \fh_0'\,|\, \beta(h_{\a_i})\in \bz,\ \forall 1\leq i\leq n\}$. An integral weight $\beta$ is called dominant if $\beta(h_{\a_i})\geq 0$ for all $1\leq i\leq n$. The weight lattice is the lattice in $\fh_0'$ generated by the fundamental weights $\omega_i$ for $1\leq i\leq n$ satisfying
$
\omega_i(h_{\a_j}) = \delta_{ij}.
$

An irreducible representation $\rho$ of $\fg$ on a finite-dimensional complex vector space $V$ has the weight space decomposition $V = \oplus V_\beta$, where $\beta\in \Lambda_W$ and 
\beq\label{wt sp}
V_\beta = \{ v\in V \,|\, \rho(H) (v) = \beta(H) v,\ \forall H\in \fh\}. 
\eeq
We have 
$
\rho(\fg_\a) V_\beta \subset V_{\a + \beta}.
$

A basic theorem \cite{Humphreys}*{\S 20} states that if $\rho$ is an irreducible representation, then there exists a unique highest weight $\nu$ (the usual notation $\lambda$ is used in this paper for the scaling factor) with a one-dimensional highest weight space $V_\nu$ such that 
\beq\label{highest}
\rho(\fn_+) V_\nu = 0. 
\eeq
All the weights of $V$ are of the form 
\beq\label{other wts}
\nu - \sum_{j=1}^n m_j \a_j,
\eeq 
where the $m_j$ are nonnegative integers. Furthermore, 
\beq\label{weyl action}
\dim V_{\tau \beta} = \dim V_\beta,\quad \tau\in W,\ \beta\text{ is a weight}.
\eeq

The Theorem of the Highest Weight \cite{Knapp}*{Theorem 5.5} asserts that up to equivalence, the irreducible finite-dimensional complex representations of $\fg$ stand in one-one correspondence with the dominant integral weights which sends an irreducible representation to its 
highest weight. We denote the irreducible representation space corresponding to a dominant weight $\nu$ by $V^\nu$. 

There is a canonical pairing between the dual space $V^*=\on{Hom}(V, \bc)$ and $V$ denoted by $\la w, v \ra \in \bc$ with $v\in V$ and $w\in V^*$. $V^*$ has a \emph{right representation} $\rho^*$ of $\fg$ defined by 
\begin{equation}\label{def right}
\la w \rho^*(X), v\ra = \la w, \rho(X)v\ra,\quad X\in \fg.
\end{equation}
 
The representation corresponding to the $i$th fundamental weight $\omega_i$ is called the $i$th fundamental representation  of $\fg$, which we denote by $V_i$. 
We choose a highest weight vector in $V_i$, and following the physicists \cite{LSbook} we called it by $|i\ra$. 
We choose a vector $\la i |$ in the lowest weight space in $V_i^*$ and require that $\la i | Id | i\ra = 1$ for the identity element $Id\in G$.  
For simplicity, we will omit the notation $\rho$ for the representation. 


Let $G^s$ be a connected and simply-connected Lie group whose Lie algebra is $\fg$. Then all the irreducible representations of $\fg$ lift to representations of $G^s$, and in particular the fundamental representations $V_i$ do. 
In this paper, we can work with a general connected Lie group $G$ whose Lie algebra is $\fg$ instead of only the simply-connected $G^s$. The reason is that 
the simply-connected compact subgroup $K^s$ of $G^s$ is used only in passing. 
Our results are expressed using $N$ and $A$ in the Iwasawa decomposition \eqref{Iwasawa}, and they are simply-connected and the same for a general $G$ and for the simply-connected $G^s$. 

There is a more concrete realization of the dual $V_i^*$ in  \ref{def right}. By the unitary trick, there exists a Hermitian form $\{\cdot, \cdot\}$ on $V_i$ (conjugate linear in the second position) invariant under the compact group $K^s$ of a simply-connected $G^s$. 
The important property of this Hermitian form is that \cite{KosToda}*{Eq. (5.11)} 
\begin{equation}\label{adjoint}
\{ g u , v \} = \{u, g^* v\}, \quad g\in G^s, \ u, v\in V_i. 
\end{equation}


Choose $v^{\omega_i}\in V_i$ to be a highest weight vector for the $i$th fundamental representation, and we require that $\{v^{\omega_i}, v^{\omega_i}\}=1$. 
Then the term in \eqref{good oldie} is, with $g=\Lambda C\Phi$,  
\begin{equation}\label{use vi}
\la i|\Phi^* C^*\Lambda^2 C \Phi|i\ra= \la i | g^* g | i \ra = \{g^* g v^{\omega_i}, v^{\omega_i}\} = \{g v^{\omega_i}, g v^{\omega_i}\} > 0.
\end{equation}

\section{Proof of the main theorem}

Now we turn to the proof of our main theorem. This proof generalizes the proof of \cite{KLNW}*{Theorem 1.1} inspired by Kostant's work \cite{KosToda}, and we need to use Proposition \ref{nonzero coeff}. 
\begin{proof}[Proof of Theorem \ref{main thm}]


We consider the main term in \eqref{mysol}. By \eqref{use vi}, 
\begin{equation}\label{big term}
\begin{split}
 \la i | \Phi^* \exp(2\lambda H) \Phi |i\ra &= \{  \exp(\lambda H) \Phi v^{\om_i}, \exp(\lambda H) \Phi v^{\om_i}\}\\
 &= \sum_{\beta, j} |\{ \exp(\lambda H) \Phi v^{\omega_i}, v^\beta_j\} |^2\\
 &= \sum_{\beta, j} |\{ \Phi v^{\omega_i}, \exp(2\lambda H) v^\beta_j\} |^2\\
 &= \sum_{\beta, j} e^{2\lambda\la \beta,H\ra}|\{ \Phi v^{\omega_i}, v^\beta_j\} |^2,
 \end{split}
 \end{equation}
 where the $\{v_j^\beta\}$ is an orthonormal basis of weight vectors of the $i$-th fundamental representation, with $v_j^\beta$  having weight $\beta$. (The possible multiplicity of the weight space with weight $\beta$ is signified by the subindex $j$.) Then the second equality follows. The third equality follows from \eqref{adjoint} and \eqref{dual} which implies that 
 $$
 \exp(\lambda H)^* = \exp(\lambda H), \quad \lambda>0, H\in \fh_0.
 $$
 The last equality follows from the definition of weight \eqref{wt sp}. 
 
 Since $\lambda\to \infty$, we first find the weight $\beta$ in the $i$th fundamental representation such that $\la \beta, H\ra$ is maximal with $H\in \tau C_0$. 
We claim that this is the case iff $\beta = \tau \omega_i$. 

Indeed, $H=\tau H_0$ with $H_0\in C_0$, then 
\beq\label{top exp}
 \la \beta, H\ra = \la \beta, \tau H_0\ra = \la \tau ^{-1}\beta, H_0\ra = \Big\la \omega_{i} - \sum_{j=1}^n m_j\alpha_{j}, H_0\Big\ra \leq \la \omega_i, H_0\ra,
\eeq
since $\tau^{-1}\beta$ is another weight in the representation $V^{\om_i}$ by \eqref{weyl action} and so it has the form by \eqref{other wts} with $m_j\geq 0$. The last inequality holds by \eqref{dom cham}, and  equality holds if and only if $\tau^{-1}\beta =\omega_i$, i.e. $\beta= \tau\omega_i$. 

Furthermore, \eqref{weyl action} says that the multiplicity of $\beta=\tau\omega_i$ is 1, and we can choose 
$$
v^{\tau\omega_i} = \bar\tau v^{\omega_i},
$$
as a basis. 

Now we need to understand the term 
$$
\{\Phi v^{\omega_i}, \bar\tau v^{\omega_i}\},
$$
using the tool from the proof of \cite{KLNW}*{Theorem 1.1}.

Let $\mathcal{S}$ be the set of all finite sequences
\begin{equation} \label{definitionofs}
s = (i_1,i_2, \cdots, i_k), \quad 1 \leq i_j \leq n,\ k \in \mathbb{Z}_{\geq 0}.
\end{equation}
We denote by $|s|$ the length $k$ of the element $s \in \mathcal{S}.$ For $s \in \mathcal{S},$ we introduce 
\begin{equation*} 
\begin{split}
\varphi(s) &= \sum_{j = 1}^{|s|} \alpha_{i_j}\in \fh_0',\\
\varphi(s,w_0) &
= \langle \varphi(s), w_0\rangle = \sum_{j = 1}^{|s|} \mu_{i_j}>0.
\end{split}
\end{equation*}
For $|s| =0,$ we define $\varphi(s) = 0.$

For $0 \leq j \leq |s| - 1,$ let $s_j \in \mathcal{S}$ be the sequence obtained from $s$ by ``cutting off" the first $j$ terms 
\cite{KosToda})
\begin{equation*} 
s_j = (i_{j + 1}, \cdots, i_{|s|}),
\end{equation*}
and define 
\begin{equation*}
\displaystyle{p(s,w_0) = \prod_{j = 0}^{|s| - 1} \langle \varphi(s_j), w_0\rangle =(\mu_{i_{1}} + \cdots +\mu_{i_{k}})\cdots(\mu_{i_{k-1}} + \mu_{i_{k}}) \mu_{i_{k}}.}
\end{equation*}
We also define $p(s, w_0) = 1$ when $|s| = 0.$ 
Let $U(\mathfrak{n}_{-})$ be the universal enveloping algebra of $\mathfrak{n}_{-}$  in \eqref{nn}. For convenience, write
$e_{-i} = e_{-\alpha_i}$ for $1\leq i\leq n.$ For $s \in \mathcal{S}$ as in \eqref{definitionofs}, define
\begin{equation*} 
e_{-s} = e_{-i_k} \cdots e_{-i_2}e_{-i_1}.
\end{equation*}

By \cite{KLNW}*{Prop. 2.2}, we have 
\begin{equation*} 
\mathcal{Y}(z) = \sum_{s \in \mathcal{S}} \frac{z^{\varphi(s,w_0)}}{p(s,w_0)}e_{-s}, \ \ \ \mbox{for} \ z \in \mathbb{C} \backslash
\mathbb{R}_{\leq 0}.
\end{equation*}
is the same as $\Phi$ in \eqref{Phi} when they act on any representation, in particular on $V^{\omega_i}$. 
Therefore, we have
\begin{equation} \label{computation2}
\left\{\Phi v^{{\omega_i}}, v^{\tau {\omega_i}}\right\} = 
\left(\sum_{s \in \mathcal{S}^{{\omega_i}}_\tau} \frac{c_{s,{\omega_i}}}{p(s,w_0)}\right) z^{\langle{\omega_i} - \tau {\omega_i}, w_0\rangle}=c\cdot z^{\langle{\omega_i} - \tau {\omega_i}, w_0\rangle},
\end{equation}
where $\mathcal{S}^{{\omega_i}}_\tau = \{s \in \mathcal{S} \,|\,  \varphi(s) = {\omega_i}- \tau {\omega_i}\},$
$c_{s, {\omega_i}} = \{e_{-s}v^{{\omega_i}}, v^{\tau {\omega_i}} \},$ and we let $c
$ denote the whole coefficient. 

Letting $z=1$ in \eqref{computation2} and by $\Phi(1)=\Psi(1)=\psi$ from Proposition \ref{alg char}, we see that  
$$
c=\{\psi v^{\om_i}, v^{\tau \om_i}\}=\{\psi v^{\om_i}, \bar \tau v^{\om_i}\}=\{(\bar \tau^*\psi) v^{\om_i}, v^{\om_i}\}.
$$
Proposition \ref{nonzero coeff} says that $\bar\tau^*\psi$ has Gauss decomposition. That is, $
\bar\tau^*\psi = n_- h n_+
$ for some $n_-\in N_-$, $h\in H$, and $n_+\in N_+$. 
Then 
\begin{equation}\label{c>0}
\begin{split}
|c| &= \big|\{(\tau^*\psi) v^{\omega_i}, v^{\omega_i}\}\big| = \big|\{ n_- h n_+ v^{\omega_i}, v^{\omega_i}\}\big|\\
&=\big|\{ h n_+ v^{\omega_i},(n_-)^* v^{\omega_i}\}\big|=\big|\{h v^{\om_i}, v^{\om_i}\}\big|>0,
\end{split}
\end{equation}
where the last equality uses that $(n_-)^*\in N_+$ and $N_+ v^{\omega_i}=v^{\omega}$ by \eqref{highest}. For the last inequality, if $h=\exp(H)$, then 
$h v^{\omega_i} = \exp(\omega_i(H)) v^{\omega_i}$, so 
$$
\big|\{h v^{\om_i}, v^{\om_i}\}\big| = \big|\exp( \omega_i(H)) \{v^{\omega_i},v^{\omega_i}\}\big| = \big|\exp( \omega_i(H))\big|>0.
$$

Putting \eqref{top exp}, \eqref{computation2} and \eqref{c>0} together, \eqref{big term} becomes
$$
 \la i | \Phi^* \exp(2\lambda H) \Phi |i\ra = e^{2\lambda\la \omega_i, H_0\ra} |c|^2 |z|^{2\la \omega_i-\tau \omega_i, w_0\ra}(1+o(1)), \quad \text{as }\lambda\to \infty,
 $$
where $o(1)$ denotes a sum of terms that have $e^{\lambda}$ to negative powers. 



Therefore, the solution in \eqref{mysol} becomes 
\begin{align*}
U_i^{\lambda H} &= 2\gamma^i \log |z| - \log\big(e^{2\lambda\la \omega_i, H_0\ra} |c|^2 |z|^{2\la \omega_i-\tau \omega_i, w_0\ra}(1+o(1)\big)\\
&= 2\gamma^i \log |z| - 2\la \omega_i-\tau \omega_i, w_0\ra\log |z| - 2\lambda \la \omega_i, H_0\ra - 2\log |c| + o(1),
\end{align*}
where $o(1)$ together with its derivatives approach 0 as $\lambda\to \infty$.

Then on the circle $S_r(0)$, using $|z|=r$, the normal derivative 
$$
\frac{\p U_i^{\lambda H}}{\p \nu} = \frac{\p U_i^{\lambda H}}{\p r} =  2\gamma^i \frac{1}{r} - 2\la \omega_i-\tau\omega_i, w_0\ra \frac{1}{r} + o(1).
$$


Applying Green's theorem to the equation \eqref{the eq}, we have 
\begin{align*}
\lim_{\lambda\to \infty} \frac{1}{\pi}\int_{B_r(0)} e^{u_i^{\lambda H}} &= -\frac{1}{4\pi} \lim_{\lambda\to \infty}\int_{B_r(0)} \Delta U_i^{\lambda H} +  \gamma^i \\
&= -\frac{1}{4\pi} \lim_{\lambda\to \infty}\int_{S_r(0)} \frac{\p U_i^{\lambda H}}{\p \nu} ds + \gamma^i\\
&= -\gamma^i + \la \omega_i-\tau\omega_i, w_0\ra + \gamma^i\\
&=   \la \omega_i-\tau\omega_i, w_0\ra .
\end{align*}
This proves the local mass formula \eqref{local mass}. 
\end{proof}



\section{An example}

We present an example for the $A_2$ Lie algebra $\fs\fl_3$ to illustrate our result. We use the standard basis for this Lie algebra, and refer the reader to any standard text on Lie algebras. For example, 
$$
h_{\alpha_1}=E_{11}-E_{22},\quad e_{-\alpha_1}=E_{21},
$$
where $E_{ij}$ is the standard elementary matrix with the only nonzero element as 1 in the $(i,j)$-position. 

For simplicity, let $\gamma_1=0, \gamma_2=0$. 
Then by \eqref{the zeta} and \eqref{Phi},  
$$
\zeta = \begin{pmatrix}
0 & & \\
1 & 0 & \\
0 & 1 & 0
\end{pmatrix},
\quad
\Phi = \begin{pmatrix}
1 & & \\
z & 1 & \\
\frac{z^2}{2} & z & 1
\end{pmatrix}.
$$

The Cartan subalgebra is 
$$
\fh_0 =  \{\on{Diag}(a_1, a_2, a_3)\,|\, a_i\in \br,\ a_1+a_2+a_3=0\}.
$$
Let $L_i\in \fh_0'$ be the linear function taking the $i$-th diagonal entry $a_i$. 

Then the simple roots and fundamental weights are 
\begin{align*}
\alpha_1 = L_1-L_2, \quad \alpha_2=L_2-L_3,\\
\omega_1=L_1,\quad \omega_2=L_1+L_2.
\end{align*}

The dominant Weyl chamber \eqref{dom cham} $C_0$ is
defined by  $a_1>a_2>a_3$. 
The Weyl group is the symmetric group $S_3$ permuting the diagonal entries. Let $\tau=(12)\in S_3$ be the the (1, 2)-transposition. 
Let $H_0=\on{Diag}(2, 1, -3)\in C_0$, then 
$$
H=\tau H_0 = \on{Diag}(1, 2, -3)\in \tau C_0.
$$
So 
$$
A := \exp(2\lambda H) = \on{Diag}(k, k^2, k^{-3}),\quad k = e^{2\lambda}.
$$

We get 
$$
\Phi^*A\Phi = \left[\begin{array}{ccc}
k+{k^{2} |z|^2}+\frac{|z|^{4}}{4 k^{3}} & {k^{2} \bar z}+\frac{z\bar z^2}{2 k^{3}} & \frac{\bar z^{2}}{2 k^{3}} 
\\
 {k^{2} z}+\frac{z^{2}\bar z}{2 k^{3}} & {k^{2}}+\frac{|z|^2}{k^{3}} & \frac{\bar z}{k^{3}} 
\\
 \frac{z^{2}}{2 k^{3}} & \frac{z}{k^{3}} & \frac{1}{k^{3}} 
\end{array}\right]
$$
in \eqref{mysol}. 

Using the first and second principal minors, we get 
\begin{align*}
U &= -\log\la 1|\Phi^*A\Phi |1\ra = -\log( k + k^2 |z|^2 + k^{-3} |z|^4/4)\\
V &= -\log\la 2|\Phi^*A\Phi | 2\ra =  -\log( k^3 + k^{-2} |z|^2 + k^{-1} |z|^4/4).
\end{align*}

Our proof of the main theorem amounts, in this case, to that 
$$
U= -\log(k^2|z|^2(1+o(1)),\quad V=-\log(k^3(1+o(1)),\quad \text{as }k\to \infty.
$$
This through Green's theorem, would imply that the masses are 
\begin{align*}
\sigma_1 &= \la \omega_1-\tau \omega_1, w_0\ra = \la \alpha_1, w_0\ra = \mu_1 = 1,\\
\sigma_2 &= \la \omega_2-\tau \omega_2, w_0\ra = \la 0, w_0\ra = 0.
\end{align*}
This is exactly what we carried out in our proof. 

To further illustrate the result, we can also compute concretely in this case that 
\begin{equation}\label{uv}
\begin{split}
u = 2U-V &= -\log \frac{( k + k^2 |z|^2 + k^{-3} |z|^4/4)^2}{( k^3 + k^{-2} |z|^2 + k^{-1} |z|^4/4)}\\
v = 2V-U &=-\log\frac{( k^3 + k^{-2} |z|^2 + k^{-1} |z|^4/4)^2}{( k + k^2 |z|^2 + k^{-3} |z|^4/4)}.
\end{split}
\end{equation}
A direct computation (using computer) shows that 
\begin{align*}
\int e^u\cdot r\,dr&=\int \frac{4 k^{4} \left(4 k^{5}+k  \,r^{4}+4 r^{2}\right)}{\left(4 k^{5} r^{2}+4 k^{4}+r^{4}\right)^{2}}\cdot r \,dr = \frac{2 k^{4} \left(-k  \,r^{2}-2\right)}{4 k^{5} r^{2}+4 k^{4}+r^{4}}+C,\\
\int e^v\cdot r\,dr&=\int \frac{4 k  \left(4 k^{5} r^{2}+4 k^{4}+r^{4}\right)}{\left(4 k^{5}+k  \,r^{4}+4 r^{2}\right)^{2}}\cdot r\,dr = \frac{-4 k^{5}-2 r^{2}}{4 k^{5}+k  \,r^{4}+4 r^{2}}+C.
\end{align*}
Therefore, 
\begin{align*}
\lim_{k\to \infty} \int_{B_r(0)} e^u &= 
\lim_{k\to \infty} 2\pi\frac{2 k^{4} \left(-k  \,r^{2}-2\right)}{4 k^{5} r^{2}+4 k^{4}+r^{4}}\bigg|^r_0 = 2\pi\Big(-\frac{1}{2} + 1\Big) = \pi,\\
\lim_{k\to \infty} \int_{B_r(0)} e^v &= \lim_{k\to \infty} 2\pi\frac{-4 k^{5}-2 r^{2}}{4 k^{5}+k  \,r^{4}+4 r^{2}}\bigg|^r_0 = 2\pi(-1+1)=0.
\end{align*}

\begin{remark}
Furthermore, we can also do the usual blowup analysis (see \cite{LYZ}). 

We see from \eqref{uv} 
$$
u(0)=\log k \to \infty,\quad v(0)\to -\infty,\quad \text{as }k\to \infty. 
$$
Let 
\begin{gather*}
\t u (z) = u(e^{-\frac{1}{2}u(0)} z) - u(0) \\
\t v (z) = v(e^{-\frac{1}{2}u(0)} z) - u(0).
\end{gather*}
We see that 
$$
\t u(z) \to -\log(1+|z|^2)^2,\quad \t v(z)\to -\infty,\qquad \text{as }k\to \infty.
$$
The first is the standard solution \eqref{liou sol} of the Liouville equation on the plane, and this also confirms that the local mass is $(1, 0)$. 
\end{remark}



\bigskip
\end{document}